\documentclass[12pt]{amsart}

\usepackage{mathrsfs}
\usepackage{graphicx}
\usepackage{amsmath}
\usepackage{amssymb}
\usepackage{amsthm}
\usepackage{graphicx}

 \DeclareMathOperator{\Var}{Var}

\newtheorem{theorem}{Theorem}
\theoremstyle{plain}

\newtheorem{definition}[theorem]{Definition}

\newtheorem{proposition}[theorem]{Proposition}
\newtheorem{corollary}[theorem]{Corollary}

\newtheorem{remark}[theorem]{Remark}

\newcommand{\n}{\noindent}

\newcommand{\R}{\mathbb{R}}
\newcommand{\A}{\mathcal{A}}
\newcommand{\E}{\mathbb{E}}
\newcommand{\B}{\mathcal{B}}

\newcommand{\mfu}{\mathfrak{f}}

\newcommand\bel[1]{\begin{equation}\label{#1}}
\newcommand\ee{\end{equation}}

\def\ds{\displaystyle}
\def\eqconst{\bowtie}

\numberwithin{equation}{section}
\numberwithin{theorem}{section}

\begin{document}

\today

\title[Parameter Estimation in Hyperbolic Equations]{Parameter Estimation in Diagonalizable Stochastic Hyperbolic Equations}
\author{W. Liu}
\curraddr[W. Liu]{Department of Mathematics, USC\\
Los Angeles, CA 90089 USA\\
tel. (+1) 213 821 1480; fax: (+1) 213 740 2424}
\email[W. Liu]{liu5@usc.edu}
\author{S. V. Lototsky}
\curraddr[S. V. Lototsky]{Department of Mathematics, USC\\
Los Angeles, CA 90089 USA\\
tel. (+1) 213 740 2389; fax: (+1) 213 740 2424}
\email[S. V. Lototsky]{lototsky@usc.edu}
\urladdr{http://www-rcf.usc.edu/$\sim$lototsky}


 \begin{abstract}
A parameter estimation problem is considered for a  linear stochastic
hyperbolic equation driven by additive space-time Gaussian white noise.
The damping/amplification operator is allowed to be unbounded.
 The estimator is of spectral type and utilizes
 a finite number of the
spatial Fourier coefficients of the solution.
The asymptotic properties of the estimator are studied
as the number of the Fourier coefficients increases, while
the observation time and the noise intensity are fixed.
\end{abstract}

\maketitle

\section{Introduction}
A typical example of a parabolic equation is the heat equation
$$
u_t=u_{xx};
$$
a typical example of a hyperbolic equation is the wave equation
$$
u_{tt}= u_{xx}.
$$
In a more abstract setting, if $\A$ is linear operator such that
$\dot{u}+\A u=0$ is a parabolic equation, then
$$
\ddot{u}+\A u=0
$$
 is natural to call a hyperbolic equation; $\dot{u}$ and  $\ddot{u}$ are the first and second time derivatives of $u$.

 {\em Damping}  in a hyperbolic equations
 is introduced via a term depending on the first time derivative of the
 solution. For example, a damped wave equation is
$$
u_{tt}= u_{xx} -a u_t, \ a>0.
$$
Indeed, if we define the total energy $E(t)=\int\big(u_t^2(t,x)+ u_x^2(t,x)\big)dx$,
then integration by parts shows that
$$
\frac{d}{dt}{E}(t)=-a\int u_t^2(t,x)dx;
$$
it also shows that $a<0$ (negative damping) corresponds to {\em amplification}.
More generally, we write a damped linear hyperbolic equation in an abstract form
\begin{equation}
\label{intr1}
\ddot{u}+\A u = \B \dot{u},
\end{equation}
where $\A$  and $\B$ are linear operators on a separable Hilbert space $H$;
depending on the
properties of the operator $\B$, the result can be either  damping or
amplification.

In this paper, we  consider a stochastic version of \eqref{intr1},
perturbed by additive space-time white noise and with operators
$\A$ and $\B$  specified up to an unknown parameter:
\begin{equation}
\label{E:3-15}
\ddot{u}+(\A_0+\theta_1\A_1)u=(\B_0+\theta_2\B_1)u_t+\dot{W}, \qquad
0<t\leq T.
\end{equation}

The objectives are
\begin{itemize}
\item to determine the conditions on the operators so that the equation
has a generalized solution that is a square-integrable random element with
values in a suitable Hilbert space;
\item to construct a maximum likelihood estimator of the unknown
parameters $\theta_1, \theta_2$ using a finite-dimensional projection of the
solution, and to study the asymptotic properties of the estimator as the dimension of the projection increases.
\end{itemize}

For stochastic parabolic equations with one unknown parameter,
  a similar problem was
first suggested by Huebner, Khasminskii and Rozovskii \cite{HKR}
 and was further investigated by Huebner and Rozovskii \cite{HubR}.
  Estimation of several parameters in parabolic
 equations has also been studied \cite{Hub1,Lot1}. For stochastic
 hyperbolic equations, most of these problems remain open. Since the
 equation is second-order in time, it is natural to start with two
 unknown parameters. In the case of the  wave equation, these
 parameters correspond to the propagation  speed of the wave and the
 damping coefficient \cite{LL}. 

 With precise definitions to come later, at this
point we interpret $\dot{W}(t)$ as a formal sum
\[
\dot{W}(t)=\sum_{k\geq1}h_k\dot{w}_k(t),
\]
where $\{h_k,\ k\geq1\}$ is an orthonormal basis in the
 Hilbert space $H$,  and $w_k(t)$ are
independent standard Brownian motions. We look for the solution of
\eqref{E:3-15} as a Fourier series
\begin{equation}
\label{intr-fk1}
u(t)=\sum_{k\geq1}u_k(t)h_k,
\end{equation}
and call it a {\tt generalized solution}.
 If the
trajectories of $u_k(t)$  are observed for $1\leq k\leq N$ and all
$0<t<T$, then there exists a closed-form expression for
maximum likelihood estimator of
$(\theta_1, \theta_2)$ in terms of $u_k$ and $\dot{u}_k$;
see Section \ref{sec:Est} below.

The main technical assumptions about the equation are
\begin{itemize}
\item zero initial conditions (to simplify the
presentation);
\item the ability to
write equation \eqref{E:3-15} as an infinite system of uncoupled stochastic
ordinary differential equation  (this is essential in the construction and the analysis of the estimator). In other words, we assume that the equation is
 {\tt diagonalizable}: the operators $\A_0$,
$\A_1$, $\B_0$ and $\B_1$ have a common system of eigenfunctions
$\{h_k,\ k\geq1\}$:
\begin{equation}
\label{eigenvalues}
\mathcal{A}_0h_k=\kappa_k h_k,\quad \mathcal{B}_0=\rho_k h_k, \quad
\mathcal{A}_1h_k=\tau_k h_k,\quad \mathcal{B}_1=\nu_k h_k
\end{equation}
and this system is  an orthonormal basis in the Hilbert space $H$.
\item {\tt hyperbolicity}, that is,
\begin{enumerate}
\item
there exist  positive numbers $C^*$, $c_1, c_2$
 such that $\{\kappa_k+\theta\tau_k+C^*,\ k\geq 1\}$ is a positive,
 non-decreasing, and unbounded  sequence
 for all $\theta\in \Theta_1$ and
 \bel{u-bnd0-intr}
c_1\leq  \frac{\kappa_k+\theta\tau_k+C^*}{\kappa_k+\theta'\tau_k+C^*}
\leq c_2
\ee
 for all $\theta, \theta'\in \Theta_1$;
 \item there exist positive numbers $C,\,J$ such that, for all $k\geq J$ and
 all $\theta_1\in \Theta_1, \ \theta_2\in \Theta_2$,
\begin{equation}
\label{eig-val-cond-intr}
T(\rho_k+\theta_2\nu_k)\leq \ln(\kappa_k+\theta_1\tau_k) + C.
\end{equation}
\end{enumerate}
\end{itemize}
If equation \eqref{E:3-15} is diagonalizable and the 
 solution has the form \eqref{intr-fk1}, then the Fourier coefficient $u_k$ satisfies
$$
\ddot{u}_k - (\rho_k+\theta \nu_k)\dot{u}_k
+(\kappa_k+\theta_1\tau_k)u_k= \dot{w}_k,\ u_k(0)=\dot{u}_k(0)=0.
$$
We show in Section  \ref{sec:Sol} that if the
equation is also hyperbolic and $X$ is a Hilbert space
such that $H\subset X$ and the embedding operator $\jmath:\ H\to X$ is
 Hilbert-Schmidt, then $u$ is an $X$-valued process.

The maximum likelihood estimators of $\theta_1$ and $\theta_2$
are constructed in Section \ref{sec:Est} using the processes
$u_k,\  \dot{u}_k,\ k=1,\ldots, N$ (the corresponding formulas are 
too complicated to present in the Introduction). 
 Analysis of these estimators in the limit $N\to \infty$
 is the main objective of the paper and
 is carried out in Sections \ref{sec:AC} and \ref{sec:GC}.
Here is the main result of the paper for the case when $\A_i, \B_i$ are
(pseudo)differential elliptic operators.

\begin{theorem}
\label{th:inter-main}
  Assume that equation \eqref{E:3-15} is diagonalizable and
hyperbolic and that  $\A_i, \B_i$ are   positive-definite
 elliptic self-adjoint differential or pseudo-differential operators  on a
 smooth bounded domain in $\R^d$ with suitable boundary conditions or on a
 smooth compact $d$-dimensional manifold. Then
 \begin{enumerate}
 \item the maximum likelihood estimator of $\theta_1$ is consistent and
 asymptotically normal in the limit $N\to \infty$ if and only if
 \begin{equation}
 \label{intr-alg-order-cond1-d}
 \mathrm{order}(\A_1)\geq \frac{\mathrm{order}(\A_0+\theta_1\A_1)
 +\mathrm{order}(\B_0+\theta_2\B_1)-d}{2};
\end{equation}
\item the maximum likelihood estimator of $\theta_2$ is consistent and
 asymptotically normal in the limit $N\to \infty$ if and only if
\begin{equation}
 \label{intr-alg-order-cond2-d}
 \mathrm{order}(\B)_1\geq \frac{\mathrm{order}(\B_0+\theta_2\B_1)-d}{2}.
 \end{equation}
 \end{enumerate}
\end{theorem}

Similar to the parabolic case (Huebner \cite{Hub1}), the results of the paper extend to a more general estimation problem
$$
\ddot{u}+\sum_{i=0}^{n}\theta_{1i}\A_iu
=\sum_{j=0}^m\theta_{2j}\B_j\dot{u}+\dot{W},
$$
as long as all the operators $\A_i, \B_j$ have a common system of
eigenfunctions. For example, in the setting similar to Theorem \ref{th:inter-main}, the coefficient $\theta_{1p}$ can be
consistently estimated if and only if
$$
\mathrm{order}(\A_p)\geq \frac{\mathrm{order}\left(\sum_{i=0}^{n}\theta_{1i}\A_i \right)
 +\mathrm{order}\left(\sum_{j=0}^m\theta_{2j}\B_j  \right)-d}{2}.
$$

Throughout the presentation below, we fix a stochastic basis
 $$
 \mathbb{F}=(\Omega, \mathcal{F},\{\mathcal{F}_t\}_{t\geq 0},
\mathbb{P})
$$
with the usual assumptions (completeness of $\mathcal{F}_0$ and
right-continuity of $\mathcal{F}_t$). We also assume that
$\mathbb{F}$ is large enough to support countably many independent
standard Brownian motions. For a random variable $\xi$, $\E \xi$ and
$\Var \xi$ denote the expectation and variance respectively.
The time derivative of a function is denote either
by  a dot on top (as in $\dot{u}$) or by a subscript $t$ 
(as in $u_t$). 

The following notations are used for two non-negative sequences
$a_n, b_n,\ n\geq 1$:
\begin{equation}
\label{eqconst}
a_n\eqconst b_n
\end{equation}
if there exist positive numbers $c_1,c_2$ such that $c_1\leq a_n/b_n \leq c_2$
for all sufficiently large $n$;
\begin{equation}
\label{not-asymp1}
a_n\asymp b_n
\end{equation}
if
\begin{equation}
\label{not-asymp1-1}
 \lim_{k\to \infty} \frac{a_k}{b_k}
=c\ \ {\rm for \ some } \ \  c>0;
\end{equation}
\begin{equation}
\label{not-asymp2}
a_n\sim b_n
\end{equation}
if \eqref{not-asymp1-1} holds with $c=1$.
 Note that if $a_n\sim b_n$ and
$\sum_n a_n$ diverges, then $\sum_{k=1}^n a_k \sim \sum_{k=1}^n b_k$.

Finally, we recall that a cylindrical Brownian motion $W=W(t)$, $t\geq 1$, over
(or on) a Hilbert space $H$ is a linear mapping
$$
W: f\mapsto W_f(\cdot)
$$
 from
$H$ to the space of zero-mean Gaussian processes such that, for
every $f,g\in H$ and $t,s>0$,
\begin{equation}
\label{CBM} \E\big(W_f(t)W_g(s)\big)=\min(t,s)(f,g)_H.
\end{equation}

A cylindrical Brownian motion $W$ is often written as a generalized
Fourier series
\begin{equation}
\label{CBM2} W(t)=\sum_{k\geq 1} w_k(t) h_k,
\end{equation}
where $w_k=W_{h_k}$. The corresponding space-time white noise is
written as
$$
\dot{W}(t)=\sum_{k\geq 1} \dot{w}_k(t)h_k.
$$

\section{Diagonalizable Stochastic Hyperbolic Equations}
\label{sec:Sol}

We start by introducing the following objects:
\begin{enumerate}
\item  $H$,  a separable Hilbert space with an orthonormal basis
$\{h_k,\ k\geq 1\}$;
\item $X$, a separable Hilbert space such that $H$ is densely and continuously
embedded into $X$ and
\begin{equation}
\label{HS-emb}
\sum_{k\geq 1} \|h_k\|_X^2<\infty
\end{equation}
 (in other words, the embedding operator from $H$ to $X$  is Hilbert-Schmidt);
\item  $\A_0,\  \A_1,\ \B_0,\ \B_1$,
linear operators on $H$;
\item $\Theta_1, \ \Theta_2$, two compact sets in $\R$;
\item $\theta_1,  \ \theta_2$, two real numbers, $\theta_1\in \Theta_1$,
$\theta_2\in \Theta_2$;
\item $(\Omega, \mathcal{F},\ (\mathcal{F}_t)_{t\geq 1}, \mathbb{P})$, a
stochastic basis with the usual assumptions and a countable collection of independent
standard Brownian motions $\{w_k=w_k(t),\ k\geq 1\}$.
\end{enumerate}

In this setting, a cylindrical Brownian motion $W=W(t)$ on $H$
is a continuous $X$-valued Gaussian process with representation
\begin{equation}
\label{cbm}
W(t)=\sum_{k\geq 1} h_k w_k(t).
 \end{equation}
The process $W$ indeed has values in $X$ rather than $H$ because 
$$
\E \|W(t)\|_X^2=
t\sum_{k\geq 1} \|h_k\|_X^2< \infty.
$$

For fixed non-random $T>0$, consider the second-order stochastic evolution equation
\begin{equation}
\label{E:3-151}
u_{tt}(t)+(\A_0+\theta_1\A_1)u(t)=(\B_0+\theta_2\B_1)u_t(t)+\dot{W}(t), \qquad 0<t\leq T,
\end{equation}
with zero initial conditions $u(0)=u_t(0)=0$.

\begin{definition}\label{D:prime}
   Equation \eqref{E:3-151} is called {\tt diagonalizable} if the operators
    $\mathcal{A}_0$, $\mathcal{A}_1$,
 $\mathcal{B}_0$, and $\mathcal{B}_1$ have a common
system of eigenfunctions $\{h_k,\ k\geq 1\}.$
\end{definition}

We will refer to $\A=\A_0+\theta_1\A_1$ and $\B=\B_0+\theta_2\B_1$
as the {\tt evolution} and {\tt dissipation} operators, respectively, and
use  notations \eqref{eigenvalues} for the
eigenvalues of the operators $\mathcal{A}_i,\  \mathcal{B}_i$.  Hyperbolicity of the equation means that the evolution operator is bounded from
below and dominates, in some sense, the dissipation operator. More precisely, we have
\begin{definition}
\label{def-parab}
A diagonalizable equation \eqref{E:3-151} is called
{\tt hyperbolic} on the time interval $[0,T]$ if
\begin{enumerate}
\item
there exist  positive numbers $C^*$, $c_1, c_2$
 such that $\{\kappa_k+\theta\tau_k+C^*,\ k\geq 1\}$ is a positive,
 non-decreasing, and unbounded  sequence
 for all $\theta\in \Theta_1$ and
 \bel{u-bnd0}
c_1\leq  \frac{\kappa_k+\theta\tau_k+C^*}{\kappa_k+\theta'\tau_k+C^*}
\leq c_2
\ee
 for all $\theta, \theta'\in \Theta_1$;
 \item there exist positive numbers $C,\,J$ such that, for all $k\geq J$ and
 all $\theta_1\in \Theta_1, \ \theta_2\in \Theta_2$,
\begin{equation}
\label{eig-val-cond}
T(\rho_k+\theta_2\nu_k)\leq \ln(\kappa_k+\theta_1\tau_k) + C.
\end{equation}
 \end{enumerate}
 \end{definition}

 Condition \eqref{eig-val-cond} means that there is no restriction on the
 strength of dissipation, but amplification must be weak.
 For example, let  $\boldsymbol{\Delta} $ be the Laplace operator in
a smooth bounded domain $G\subset \R^d$ with zero boundary conditions, and
$H=L_2(G)$.
 Then each of the following equations is diagonalizable and hyperbolic
 on $[0,T]$ for all $T>0$:
\begin{equation}
\label{example-eq}
\begin{split}
u_{tt}&=\boldsymbol{\Delta} u + u_t + \dot{W}, \ u_{tt}=\boldsymbol{\Delta} u- u_t+\dot{W},\\
u_{tt}&=\boldsymbol{\Delta} (u+u_t)+\dot{W},\ u_{tt}=\boldsymbol{\Delta} u-\boldsymbol{\Delta}^2u_t+\dot{W},
\end{split}
\end{equation}
while equations
$$
u_{tt}=\boldsymbol{\Delta} (u-u_t)+\dot{W},\ u_{tt}=\boldsymbol{\Delta} u+\boldsymbol{\Delta}^2u_t+\dot{W}
$$
 are diagonalizable but not hyperbolic on any $[0,T]$.
 To construct an example of an equation that is hyperbolic on every time interval $[0,T]$ and has unbounded amplification, take $\theta_1=\theta_2=1$ and 
 consider  the operators  with eigenvalues $\kappa_k= \rho_k=0$, $\tau_k=e^k$, $\nu_k=\ln k$.
  
 The following result shows that, in a hyperbolic equation, the evolution operator
 is uniformly bounded from below.
 \begin{proposition}
 \label{prop-ubnd}
 If equation \eqref{E:3-151} is diagonalizable and hyperbolic, then
 \bel{u-unbd}
 \lim_{k\to \infty}\big(\kappa_k+\theta\tau_k\big)=+\infty
 \ee
  uniformly in $\theta \in \Theta_1,$ and there
 exists an index $J\geq 1$ and a number $c_0$ such that, for all $k\geq J$ and
 $\theta\in \Theta_1$,
 \begin{align}
 \label{pos-mu}
 &\kappa_k+\theta\tau_k>1,\\
 \label{u-bnd}
 &\frac{|\tau_k|}{\kappa_k+\theta\tau_k} \leq c_0.
 \end{align}
 \end{proposition}

 \begin{proof} To simplify the notations, define
 $$
 \lambda_k(\theta)=\kappa_k+\theta\tau_k.
 $$
 Since  $\{\lambda_k(\theta)+C^*,\ k\geq 1\}$ is a positive,
 non-decreasing, and unbounded  sequence
 for all $\theta\in \Theta_1$ and \eqref{u-bnd0} holds,
  we have \eqref{u-unbd}, and then
 \eqref{pos-mu} follows.

 To prove \eqref{u-bnd}, we argue by contradiction.
 Assume that the sequence
 $\{|\tau_k|\,\lambda^{-1}_k(\theta),k\geq 1\} $ is not uniformly bounded. Then
 there is a sequence $ \{|\tau_{k_j}|\,\lambda^{-1}_{k_j}(\theta_j),\ j\geq 1\}$
 such that
 \bel{u-bnd-pr1}
 \lim_{j\to \infty} \frac{|\tau_{k_j}|}{\theta_{j}\tau_{k_j}+\kappa_{k_j}}=+\infty.
 \ee
 With no loss of generality, assume that $\tau_{k_j}>0$, and,
 since $\Theta_1$ is compact, we also assume that
 $\lim_{j\to \infty} \theta_j=\theta^{\circ}\in \Theta_1$
 (if not, extract a further sub-sequence).

 Then \eqref{u-bnd-pr1} implies
 \bel{u-bnd-pr2}
 \lim_{j\to \infty} \frac{\kappa_{k_j}}{\tau_{k_j}}=-\theta^{\circ}.
 \ee
 Note that $\lim_{j\to \infty} |\tau_{k_j}|=+\infty$, because
 $\lim_{j\to \infty}(\theta^{\circ}\tau_{k_j}+\kappa_{k_j})=+\infty$.
 Consequently,
 $$
 \lim_{j\to \infty}
 \frac{\lambda_{k_j}(\theta)+C^*}{\lambda_{k_j}(\theta^{\circ})+C^*}=
 \frac{\theta-\theta^{\circ}}
 {\theta^{\circ}+\lim_{j\to \infty} ({\kappa_{k_j}}/{\tau_{k_j}})}=\infty,\ \
 \theta\not=\theta^{\circ}.
 $$
 As a result,  if \eqref{u-bnd} fails, then so does
  \eqref{u-bnd0} for  $\theta\not=\theta^{\circ}$,
 $\theta'=\theta^{\circ}$.
 \end{proof}

To state the result about existence and uniqueness of solution for
\eqref{E:3-151}, note that we do not have enough information about the
operators $\A_i$ and $\B_i$ to define the traditional variational solution
because we are not assuming that the operators act in a normal triple of Hilbert spaces --- the usual setting to define
a variational solution (see, for example, Chow \cite[Section 6.8]{Chow-Spde}).
On the other hand,  if the
operators $\A_i, \B_i$ {\em were} bounded and if the process
$W$ {\em were} $H$-valued, then there would be a unique process
$v=v(t)$ with continuous trajectories in $H$ such that
$u(t)=\int_0^t v(s)ds$ and
$$
v(t)+\int_0^t (\A_0+\theta_1 \A_1)u(s)ds=\int_0^t (\B_0+\theta_2 \B_1)v(s)ds+W(s);
$$
 see  Chow \cite[Theorem 6.8.2]{Chow-Spde}.
If, in addition, equation \eqref{E:3-151} is diagonalizable,  then $u$ would have  the following expansion in the basis $\{h_k,\ k\geq 1\}$:
\begin{align}
& u(t)=\sum_{k\geq 1} u_k(t)h_k, \label{sol-FSexp} \\
\ddot{u}_k(t)&-(\rho_k+\theta_2\nu_k)\dot{u}_k(t)
+(\kappa_k+\theta_1\tau_k)u_k(t)=
\dot{w}_k(t),\ u_k(0)=\dot{u}_k(0)=0.\label{sol-coef-eq}
\end{align}
The basis $\{h_k,\ k\geq 1\}$ thus becomes a natural collection of 
test functions. 

Since the operators $\A_i,\ \B_i$ are in general not bounded on $H$ and the
process $W$ is not $H$-valued, we use the auxiliary space $X$ and
establish the following result.

\begin{theorem}
\label{th:main1}
Assume that equation \eqref{E:3-151} is diagonalizable and hyperbolic.
Then there is a unique adapted
$X$-valued process $u=u(t)$ with representation \eqref{sol-FSexp},
\eqref{sol-coef-eq}; we call the process $u$ a {\tt ge\-ne\-ra\-lized solution} of
\eqref{E:3-151}. If, in addition, there exists a real number
$C_0$ such that  $\theta_2\nu_k+\rho_k\leq C_0$ for all $k$, then
$v(t)=\dot{u}(t)$ is also an $X$-valued process.
\end{theorem}

Let us make a few comments about the result.
\begin{enumerate}
\item By Lemma \ref{prop-ubnd}, we know that $\kappa_k+\theta_1\tau_k>1$ for
all sufficiently large $k$. Condition \eqref{eig-val-cond} means certain
subordination of the dissipation operator $\B_0+\theta_2\B_1$ to the
evolution operator $\A_0+\theta_1\A_1$. In particular, any dissipation
(negative $\mu_k=\rho_k+\theta_2\nu_k$) is admissible, as well as certain unbounded amplification
(positive and unbounded $\mu_k$), as long as the sequence $\{\mu_k,\ k\geq 1\}$ does not grow too fast; the critical growth rate depends on the length of the time interval. This possibility
to have an {\em unbounded dissipation operator} makes the result 
different from those considered in the literature, such as 
  \cite[Theorem 6.8.4]{Chow-Spde}.
\item The resulting generalized solution is {\em weak in the PDE sense, but is strong in the probabilistic sense}, being constructed on a given stochastic
basis;
\item
The solution is defined by its Fourier coefficients and therefore does not
depend on the choice of the space $X$. The role of $X$ is to ensure
that the  equation is  {\em  well-posed} in the sense that the output process 
(the solution $u$) takes values in  the same space as the ``input'' process $W$.
Given the special form of $W$, we are not discussing any {\em continuous 
dependence} of $u$ on $W$. 
\end{enumerate}

\begin{proof}[Proof of Theorem \protect{\ref{th:main1}}.]
 To simplify the presentation, introduce the notations 
\begin{equation}
\label{eq:lambda-mu}
\lambda_k=\kappa_k+\theta_1\tau_k,\ \ \mu_k=\rho_k+\theta_2\nu_k.
\end{equation}

For a fixed $k\geq 1$, let us consider the process $u_k$ defined by
\eqref{sol-coef-eq}. Equation \eqref{sol-coef-eq} has a unique solution, and direct computations show that
\begin{equation}
\label{coef-expl-form}
u_k(t)=\int_0^t \mfu_k(t-s)dw_k(s),
\end{equation}
where the {\tt fundamental solution} $\mfu_k$ satisfies
\begin{equation}
\label{fund-sol}
\ddot{\mfu}_k(t)-\mu_k\dot{\mfu}_k(t)+\lambda_k\mfu_k(t)=0,\ \
\mfu_k(0)=0,\ \dot{\mfu}_k(0)=1;
\end{equation}
see Appendix for details. 
Thus, $\E u_k(t)=0$ and, since the processes $u_k$ are independent for different
$k$, the series \eqref{sol-FSexp} defines an $X$-valued process if 
\begin{equation}
\label{coef-bnd}
\sup_{k\geq 1} \sup_{t\in [0,T]} \E|u_k(t)|^2 < \infty.
\end{equation}
By direct computation using \eqref{coef-expl-form} and the It\^{o} isometry,
\begin{equation}
\label{coef-sq}
\E |u_k(t)|^2 = \int_0^t \mfu_k^2(t-s)ds=\int_0^t \mfu_k^2(s)ds.
\end{equation}
The proof of the theorem is thus reduced to the study of the fundamental solution
$\mfu_k$ for sufficiently large $k$. More precisely, we will show that
\begin{equation}
\label{FS-bound}
\sup_{t\in [0,T]} \sup_k \mfu_k^2(t) < \infty,
\end{equation}
which, by \eqref{coef-sq}, implies \eqref{coef-bnd}.

The solution of equation \eqref{fund-sol} is determined by the
roots $r_{\pm}$ of the characteristic equation
\begin{equation}
\label{char-eq}
r^2-\mu_k r +\lambda_k=0: \ \ r_{\pm}=\frac{\mu_k\pm\sqrt{\mu_k^2-4\lambda_k}}{2}.
\end{equation}
By Lemma \ref{prop-ubnd}, $\lim_{k\to \infty}\lambda_k=+\infty,$ and,
 in particular, $\lambda_k>0$ for all sufficiently large $k$.
Also, condition \eqref{eig-val-cond} means that if $\mu_k>0$, then
$\mu_k\leq (\ln \lambda_k+C)/T$, and therefore 
 $\mu_k<2\sqrt{\lambda_k}$ for all sufficiently large $k$.
Accordingly, we assume that $\lambda_k>0$ and consider two cases: $|\mu_k|<2 \sqrt{\lambda_k}$ and
$\mu_k\leq - 2\sqrt{\lambda_k}$.

{ If $|\mu_k|<2\sqrt{\lambda_k}$, then} equation \eqref{char-eq} has complex
conjugate roots, and, with $\ell_k=\sqrt{\lambda_k-(\mu_k^2/4)}$,
\begin{equation}
\label{fs-compl}
\mfu_k^2(t)=t^2 e^{\mu_k t} \left(\frac{\sin(\ell_k t)}{\ell_k t}\right)^2.
\end{equation}
If $\mu_k\leq 0$, then $\mfu_k^2\leq T^2$ for all $t\in [0,T]$ and
 \eqref{FS-bound} follows.
 If $\mu_k>0$, then, for
sufficiently large $k$, condition \eqref{eig-val-cond} ensures that
$e^{\mu_k t} \leq \lambda_k e^C$ and  $\lambda_k/\ell_k^2 < 2$.
Then $\mfu_k^2\leq 2T^2e^C$ and \eqref{FS-bound} follows.

{ If $\mu_k\leq -2\sqrt{\lambda_k}$, then}  \eqref{char-eq} has real roots
(a double root if $\mu_k=-2\sqrt{\lambda_k}$), and, using the notations
$\ell_k=\sqrt{\mu_k^2-4\lambda_k}$, $a=\mu_k+\ell_k$,
\begin{equation}
\label{fs-real}
\mfu_k^2(t)=t^2e^{at}\left(\frac{1-e^{-\ell_k t}}{\ell_k t}\right)^2;
\end{equation}
the case of the double root corresponds to the limit $\ell_k\to 0$.
By assumption, $a\leq 0$, so that
 $\mfu_k^2(t) \leq T^2$ and \eqref{FS-bound} follows.

 Similarly, $v(t)=\sum_k v_k(t)h_k$,
 $v_k(t)=\int_0^t \dot{\mfu}_k(t-s)dw_k(s)$, and
 $\E v_k^2 (t)=\int_0^t |\dot{\mfu}_k(s)|^2ds$. By direct computation,
  if $\mu_k \leq C_0$, then
 $\sup_{t\in [0,T]} \sup_k |\dot{\mfu}_k(t)|^2 < \infty$, and therefore
 $v(t)\in X$.

This completes the proof of Theorem \protect{\ref{th:main1}}.
\end{proof}

\section{Estimation of Parameters}
\label{sec:Est}

Assume that the solution of equation \eqref{E:3-151} is observed so that the
 measurements of $u_k(t)$ and $v_k(t)=\dot{u}_k$ are available for all $t\in
[0,T]$ and $k=1,\ldots, N$. The objective is to estimate the
parameters $\theta_1, \theta_2$.
We keep notations \eqref{eigenvalues}, and also define
\begin{equation}
\label{eq:lambda-mu1}
\lambda_k(\theta)=\kappa_k+\theta\tau_k,\ \ \mu_k(\theta)=\rho_k+\theta\nu_k.
\end{equation}

Since
\begin{equation}
\label{E:3-22}
dv_k(t)=\big(-\lambda_k(\theta_1)u_k+\mu_k(\theta_2)v_k
\big)dt+dw_k(t)
\end{equation}
(see \eqref{sol-coef-eq}), and $u_k(t)=\int_0^t v_k(s)ds$, the vector process $\mathbf{v}=(v_1, \ldots, v_N)$ is a diffusion-type process in the sense of Liptser and
Shiryaev; see  \cite[Definition 4.2.7]{LSh1}. Therefore, by
Theorem 7.6 in \cite{LSh1} (see also Section 7.2.7 of the same reference),
the measure $\mathbf{P}^{\mathbf{v}}$ generated by the process
$\mathbf{v}$ in the space of $\mathbb{R}^N$-valued continuous
functions on $[0,T]$ is absolutely continuous with respect to the measure
$\mathbf{P}^{\mathbf{w}}$, generated in the same space by the
$N$-dimensional standard Brownian motion $\mathbf{w}=(w_1, \ldots, w_N)$.
Moreover, the density $Z=d\mathbf{P}^{\mathbf{v}}/d\mathbf{P}^{\mathbf{w}}$ has a representation \begin{equation}
\begin{split}
\label{density}
Z(\mathbf{v})&=\exp \Bigg(\sum_{k=1}^N\Big(\int_0^T
\big(-\lambda_k(\theta_1)u_k(t)+\mu_k(\theta_2)v_k(t) \big)dv_k(t)
\notag\\
&-\frac{1}{2}\int_0^T\big(-\lambda_k(\theta_1)u_k(t)+\mu_k(\theta_2)v_k(t)
\big)^2dt\Big) \Bigg).
\end{split}
\end{equation}
Define
$$
\mathfrak{z}=\ln Z(\mathbf{v}).
$$
Note that $\mathfrak{z}$ is a function of $\theta_1, \theta_2$, and the maximum
likelihood estimator of the parameters $\theta_1, \theta_2$ is computed  by
solving the system of equations
\begin{equation}
\label{mle1}
\frac{\partial \mathfrak{z}}{\partial \theta_1}=0,\ \ \
\frac{\partial \mathfrak{z}}{\partial \theta_2}=0,
\end{equation}
with unknowns $\theta_1,\theta_2$. 
This system can be written as 
 \begin{equation}
 \label{mle2}
 \begin{split}
 &F_{1,N}+L_{1,N}+K_{1,N}\theta_1+K_{12,N}\theta_2=A_{1,N}\\
 &F_{2,N}+L_{2,N}+K_{12,N}\theta_1+K_{2,N}\theta_2=A_{2,N},
 \end{split}
 \end{equation}
where 
\begin{equation}
\label{MainNotations}
\begin{split}
& A_{1,N}=-\sum_{k=1}^N\int_0^T{\tau_k u_k(t)dv_k(t)}, \quad
A_{2,N}=\sum_{k=1}^N\int_0^T{\nu_k v_k(t)dv_k(t)},\\
 &F_{1,N}=-\sum_{k=1}^N\int_0^T\kappa_k\tau_ku^2_k(t)dt, \quad
F_{2,N}=\sum_{k=1}^N\int_0^T\rho_k\nu_kv^2_k(t)dt, \\
&K_{1,N}=\sum_{k=1}^N\int_0^T\!\!\!\tau^2_ku^2_k(t)dt,\ 
K_{2,N}=\sum_{k=1}^N\int_0^T\!\!\!\nu^2_kv^2_k(t)dt,\ K_{12,N}=-\sum_{k=1}^N\int_0^T\!\!\!\nu_k\tau_ku_k(t)v_k(t)dt, \\
&L_{1,N}=-\sum_{k=1}^N\int_0^T\rho_k\tau_ku_k(t)v_k(t)dt,\quad
L_{2,N}=-\sum_{k=1}^N\int_0^T\kappa_k\nu_ku_k(t)v_k(t)dt. \\
\end{split}
\end{equation}
All the numbers $A$, $F$, $L$ and $K$ are computable from the
observations of $u_k(t)$ and $v_k(t)$, $k=1,\ldots, N$, $t\in [0,T].$

Note that
$$
K_{12,N}=-\frac{1}{2}\sum_{k=1}^N\tau_k\nu_ku^2_k(T),\quad
L_{1,N}=-\frac{1}{2}\sum_{k=1}^N\rho_k\tau_ku^2_k(T),\quad
L_{2,N}=-\frac{1}{2}\sum_{k=1}^N \kappa_k\nu_ku^2_k(T),
$$
 because, by assumption, $u_k(0)=0$ and thus
$$
\int_0^T u_kv_k(t)dt=\int_0^T u_k(t)du_k(t)=\frac{1}{2} u_k^2(T).
$$

 By the Cauchy-Schwartz inequality,
$K_{1,N}K_{2,N}-K^2_{12,N}>0$ with probability one, because
the process $u_k$ is not a scalar multiple of $v_k.$
 Therefore \eqref{mle2} has  a unique solution 
\begin{equation}
\label{E:18}
\begin{split}
\hat{\theta}_{1,N}&=
\frac{K_{2,N}\big(A_{1,N}-F_{1,N}-L_{1,N}\big)
-K_{12,N}\big(A_{2,N}-L_{2,N}-F_{2,N}\big)}{K_{1,N}K_{2,N}-K^2_{12,N}},
\\
\hat{\theta}_{2,N}&=
\frac{K_{1,N}\big(A_{2,N} -F_{2,N} -L_{2,N}\big)-K_{12,N}\big(A_{1,N}-F_{1,N}-L_{1,N}\big)}
{K_{1,N}K_{2,N}-K^2_{12,N}}.
\end{split}
\end{equation}

With notations \eqref{MainNotations} in mind, formulas \eqref{E:18} provide 
explicit expressions for the maximum likelihood estimators of
$\theta_1$ and $\theta_2$. 
To study asymptotic properties of these estimators, we need
expressions for $\hat{\theta}_{i,N}-\theta_i$, $i=1,2$:
\begin{equation}
\label{mle55}
\begin{split}
\hat{\theta}_{1,N}-\theta_1&=\frac{1}{1-D_N}\left(\frac{\iota_{1,N}}{K_{1,N}}-
\frac{\iota_{2,N}K_{12,N}}{K_{1,N}K_{2,N}}\right),\\
\hat{\theta}_{2,N}-\theta_2&=\frac{1}{1-D_N}\left(\frac{\iota_{2,N}}{K_{2,N}}-
\frac{\iota_{1,N}K_{12,N}}{K_{1,N}K_{2,N}}\right),
\end{split}
\end{equation}
 where 
\begin{equation}
\iota_{1,N}=-\sum_{k=1}^N\int_0^T\tau_ku_k(t)dw_k(t),\quad
\iota_{2,N}=\sum_{k=1}^N\int_0^T\nu_kv_k(t)dw_k(t), \quad
D_N=\frac{K_{12,N}^2}{K_{1,N}K_{2,N}}.
\end{equation}

It follows that, as $N\to \infty$, asymptotic behavior of the estimators is determined by $\iota_{i,N}/K_{i,N}$, $i=1,2$, and ${K_{12,N}}/{(K_{1,N}K_{2,N})}$.
Note that each of $\iota_{i,N}, K_{i,N}, K_{12,N}$ is a sum of independent
random variables. Moreover,
\begin{equation}
\label{iota}
\E \iota_{i,N}^2=\E K_{i,N},\ i=1,2.
\end{equation}

If $\mfu_k$ is the function satisfying
\begin{equation}
\label{fund-sol1}
\ddot{\mfu}_k(t)-\mu_k(\theta_2)\dot{\mfu}_k(t)+\lambda_k(\theta_1)
\mfu_k(t)=0,\ \
\mfu_k(0)=0,\ \dot{\mfu}_k(0)=1,
\end{equation}
then, by direct computation,  $u_k(t)=\int_0^t \mfu_k(t-s)dw_k(s)$
(see Appendix for more details), so that
$$
\E u_k^2(t)=\int_0^t |\mfu_k(s)|^2ds,\ \E v_k^2(t)=\int_0^t |\dot{\mfu}_k(s)|^2ds,
$$
and
\begin{equation}
\label{eq:rates00}
\begin{split}
&\Psi_{1,N}:=\E K_{1,N}=\sum_{k=1}^N\tau^2_k\int_0^T\int_0^t |\mfu_k(s)|^2dsdt,\\
&\Psi_{2,N}:=\E K_{2,N}=\sum_{k=1}^N\nu^2_k\int_0^T\int_0^t |\dot{\mfu}_k(s)|^2dsdt,\\
&\Psi_{12,N}:=\E K_{12,N}=-\frac{1}{2}\sum_{k=1}^N
\tau_k\nu_k\int_0^T |\mfu_k(s)|^2ds.
\end{split}
\end{equation}

The following is a necessary conditions for the consistency
of the estimators.
\begin{proposition}
\label{prop:necessary}
If  $\lim_{N\to \infty} \hat{\theta}_{1,N}=\theta_i$ in probability,
 then $\lim_{N\to \infty}\Psi_{1,N}=+\infty$. Similarly,
 if  $\lim_{N\to \infty} \hat{\theta}_{2,N}=\theta_i$ in probability,
 then $\lim_{N\to \infty}\Psi_{2,N}=+\infty$.

\end{proposition}

\begin{proof} Each of the sequences $\{\Psi_{i,N},\ N\geq 1\}$ is
monotonically increasing and thus has a limit, finite or infinite.
If $\lim_{N\to \infty}\Psi_{i,N}<\infty$, then
$\lim_{N\to \infty} \iota_{i,N}/K_{i,N}$ exists with probability one and is
a non-degenerate random variable. Equalities  \eqref{mle55} then
implies that $\hat{\theta}_{i,N}$ cannot converge to $\theta_i$.
\end{proof}

Under the assumptions of Theorem \ref{th:main1}, we derived a bound
$|\mfu_k(t)|^2\leq const. \cdot T^2$, which was enough to establish
existence and uniqueness of solution of \eqref{E:3-151}. To study estimators
$\hat{\theta}_{i,N}$, and, in particular, convergence/divergece of the
sequences $\{\Psi_{i,N},\ N\geq 1\}$, we need more delicate bounds on both
$|\mfu_k(t)|^2$ and $|\dot{\mfu}_k(t)|^2$. The computations,
while relatively straightforward, are rather long and
lead to the following relations (see \eqref{not-asymp1} for the 
definition of $\sim$):
\begin{align}
\label{main-asympt-u0}
&\E u_k^2(T)\sim
\frac{e^{\mu_k(\theta_2)T}-1}{2\mu_k(\theta_2)\lambda_k(\theta_1)}, \quad
 \Var u_k^2(T)\sim 3 \left(\frac{e^{\mu_k(\theta_2)T}-1}{2\mu_k(\theta_2)\lambda_k(\theta_1)}
 \right)^2;\\
\label{main-asympt-u}
&\E \int_0^T u_k^2(t)dt \sim \frac{T^2M\big(T\mu_k(\theta_2)\big)}{\lambda_k(\theta_1)},\quad
\Var \int_0^T u_k^2(t)dt \sim \frac{T^4V\big(T\mu_k(\theta_2)\big)}{\lambda_k^2(\theta_1)},\\
\label{main-asympt-v}
&\E \int_0^T v_k^2(t)dt \sim T^2M\big(T\mu_k(\theta_2)\big),\quad
\Var \int_0^T v_k^2(t)dt \sim T^4V\big(T\mu_k(\theta_2)\big),
\end{align}
where
\begin{align}
\label{aux-M}
&M(x)=
\begin{cases}
\ds \frac{e^{x}-x-1}{2x^2},&{\ \rm if\ } x\not= 0,\\
\ds \frac{1}{4},& {\ \rm if}\   x=0;
\end{cases}\\
\label{aux-V}
&V(x)=
\begin{cases}
\ds\frac{e^{2x}+4e^{x}-4xe^{x}-2x-5}{4x^4},&{\ \rm if\ } x\not= 0,\\
\ds\frac{1}{24},& {\ \rm if}\   x=0.
\end{cases}
\end{align}
Note that the functions
$M$ and $V$ are continuous and positive on $\R$, and
\begin{equation}
\label{MV-asympt}
M(x)\sim
\begin{cases}
\ds (2|x|)^{-1},&  x\to -\infty,\\
\ds 2(2x)^{-2}\ e^x,& x\to +\infty;
\end{cases}
\  \ V(x)\sim
\begin{cases}
\ds 4(2|x|)^{-3},&  x\to -\infty,\\
\ds 4(2x)^{-4}\ e^{2x},& x\to +\infty.
\end{cases}
\end{equation}

 The computations  leading to \eqref{main-asympt-u0}--\eqref{main-asympt-v}
  rely on the fact that
 $u_k$ and $v_k$ are Gaussian processes, so that, for example,
 $$
 \Var  \int_0^T u_k^2(t)dt = 4\int_0^T \int_0^t \Big( \E \big(u(t) u(s) \big) \Big)^2 dsdt.
 $$

 It follows from \eqref{main-asympt-u} and \eqref{main-asympt-v} that if $\lim_{N\to \infty} \Psi_{i,N}=+ \infty$, then
\begin{equation}
\label{eq:rates1}
\Psi_{1,N}\sim T^2\sum_{k=1}^N
\frac{\tau_k^2\,M\big(T\mu_k(\theta_2)\big)}{\lambda_k(\theta_1)},\
\Psi_{2,N}\sim T^2\sum_{k=1}^N
\nu_k^2\,M\big(T\mu_k(\theta_2)\big).
\end{equation}
Relations \eqref{eq:rates1} show that
conditions for consistency and asymptotic normality
of the estimators require
additional assumptions on the asymptotical behavior of the
eigenvalues of the operators
$\A_i$, $\B_i$.

The asymptotic behavior of the eigenvalues of an operator is well-known
when the operator is elliptic and self-adjoint. For example,
let $\mathcal{D}$ be an operator defined on smooth functions by
$$
\mathcal{D}f(x)=-\sum_{i,j=1}^d \frac{\partial }{\partial x_i}\left(a_{ij}(x)
\frac{\partial f(x)}{\partial x_j}\right),
$$
in a smooth bounded domain $G\subset \R^d$, with zero Dirichlet boundary
conditions. Assume that the functions $a_{ij}$ are all infinitely differentiable
in $G$ and are bounded with all the derivatives, and the matrix
$(a_{ij}(x), \ i,j=1,\ldots, d)$ is symmetric and uniformly positive-definite for all
$x\in G$. Then the eigenvalues $d_k$ of $\mathcal{D}$ 
can be enumerated so that
\begin{equation}
\label{asympt-eig00}
d_k\asymp k^{2/d}
\end{equation}
in the sense of notation \eqref{not-asymp1}.
 More generally, for a positive-definite
 elliptic self-adjoint differential or pseudo-differential operator $\mathcal{D}$
 of order $m$ on a
 smooth bounded domain in $\R^d$ with suitable boundary conditions or on a
 smooth compact $d$-dimensional manifold, the asymptotic of the eigenvalues
 $d_k,\ k\geq 1$, is
 \begin{equation}
 \label{asympt-eig01}
d_k\asymp k^{m/d};
\end{equation}
note that $m$ can be an arbitrary positive number. This result is well-known;
see, for example,  Safarov and Vassiliev \cite[Section 1.2]{SafVas}.
 An  example of $\mathcal{D}$ is $(1-\boldsymbol{\Delta})^{m/2}$, $m>0$, where $\boldsymbol{\Delta}$ is the Laplace operator; note also that, for this operator, relation \eqref{asympt-eig01} holds even when $m\leq 0$.

In our setting, when the operators are {\em defined} by their eigenvalues and
eigenfunctions, more exotic eigenvalues are possible, for example,
$\tau_k=e^k$ or $\nu_k=(-1)^k/k$. On the other hand, it is clear that the
analysis of the estimators should be easier when all the eigenvalues in the
equation are of the type \eqref{asympt-eig01}. Accordingly, we make the
following
\begin{definition}
\label{def:algebr}
Equation \eqref{E:3-151} is called {\tt algebraically hyperbolic}
if it is diagonalizable, hyperbolic, and the eigenvalues
$\lambda_k(\theta)=\kappa_k+\theta\tau_k$, $\mu_k(\theta)=\rho_k+\theta\nu_k$
have the following properties:
\begin{enumerate}
\item There exist real
numbers $\alpha, \alpha_1$ such that, for all $\theta\in \Theta_1$,
\begin{equation}
\label{eq:alg-cond1}
\lambda_k(\theta) \asymp k^{\alpha},\ \ |\tau_k|\asymp k^{\alpha_1};
\end{equation}

\item Either $|\mu_k(\theta)|\leq C$ for all $\theta\in \Theta_2$ or
there exist  numbers $\beta>0,\ \beta_1\in \R$ such that,
for all $\theta\in \Theta_2$,
\begin{equation}
\label{eq:alg-cond2}
-\mu_k(\theta) \asymp k^{\beta},\ |\nu_k|\asymp k^{\beta_1}.
\end{equation}
\end{enumerate}
\end{definition}
To emphasize the importance of the numbers $\alpha$ and $\beta$,
we will sometimes say that the equation is $(\alpha,\beta)$-algebraically
hyperbolic; $\beta=0$ includes the case of uniformly bounded $\mu_k(\theta)$.

The reader can easily verify that
\begin{itemize}
\item under hyperbolicity assumption, $\alpha>0$ and no unbounded
amplification is possible;
\item each of the equations in \eqref{example-eq} is algebraically hyperbolic.
\end{itemize}

\section{Analysis of Estimators: Algebraic Case}
\label{sec:AC}

\begin{theorem}
\label{th:main-alg}
Assume that
equation \eqref{E:3-151} is $(\alpha, \beta)$-algebraically
hyperbolic in the sense of Definition \ref{def:algebr}.
\begin{enumerate}
\item If
\begin{equation}
\label{alg-order-cond1}
\alpha_1\geq \frac{\alpha+\beta-1}{2},
\end{equation}
then the estimator $\hat{\theta}_{1,N}$ is strongly consistent and
asymptotically normal with rate $\sqrt{\Psi_{1,N}}$ as $N\to \infty$:
\begin{align}
\label{alg-res-asymp1}
\lim_{N\to \infty}\hat{\theta}_{1,N}=\theta_1 \ \ {\rm with\ probability \ one}; \\
\label{alg-res-asymp2}
\lim_{N\to \infty}\sqrt{\Psi_{1,N}}\Big(\hat{\theta}_{1,N}-\theta_1 \Big)
= \xi_1 \ {\rm in \ distribution,}\
\end{align}
where $\xi_1$ is a standard Gaussian random variable.
\item If
\begin{equation}
\label{alg-order-cond2}
\beta_1\geq \frac{\beta-1}{2},
\end{equation}
then the estimator $\hat{\theta}_{2,N}$ is strongly consistent and
asymptotically normal with rate $\sqrt{\Psi_{2,N}}$ as $N\to \infty$:
\begin{align}
\label{alg-res-asymp3}
\lim_{N\to \infty}\hat{\theta}_{2,N}=\theta_2 \ \ {\rm with\ probability \ one}; \\
\label{alg-res-asymp4}
\lim_{N\to \infty}\sqrt{\Psi_{2,N}}\Big(\hat{\theta}_{2,N}-\theta_2 \Big)
= \xi_2 \ {\rm in \ distribution,}\
\end{align}
where $\xi_2$ is a standard Gaussian random variable.
\item If both \eqref{alg-order-cond1} and \eqref{alg-order-cond2} hold,
then the random variables $\xi_1, \xi_2$ are independent.
\end{enumerate}
\end{theorem}

\begin{remark}
(a)  In terms of the orders of the operators $($see \eqref{asympt-eig01}$)$,
  condition \eqref{alg-order-cond1} becomes
 \begin{equation}
 \label{alg-order-cond1-d}
 \mathrm{order}(\A_1)\geq \frac{\mathrm{order}(\A_0+\theta_1\A_1)
 +\mathrm{order}(\B_0+\theta_2\B_1)-d}{2},
\end{equation}
and condition \eqref{alg-order-cond2} becomes
\begin{equation}
 \label{alg-order-cond2-d}
 \mathrm{order}(\B_1)\geq \frac{\mathrm{order}(\B_0+\theta_2\B_1)-d}{2}.
\end{equation}
(b) The condition for consistency of $\hat{\theta}_{2,N}$
does not depend on the evolution operator and is similar to the
consistency condition in the parabolic case  \cite[Theorem 2.1]{HubR}.
\end{remark}

The intuition behind conditions \eqref{alg-order-cond1-d} and
\eqref{alg-order-cond2-d} is as follows. The information about the
numbers $\theta_1, \theta_2$ is carried by the terms
$\A_1u$ and $\B_1 u_t$, respectively, and these terms must be
irregular enough to be distinguishable in the noise $\dot{W}$ {\em during
 a finite observation window $[0,T]$}.
The higher the orders of the operators, the more irregular the
terms, the easier the estimation.

\begin{proof}[Proof of Theorem \protect{\ref{th:main-alg}}]
Note that if $\beta>0$, then $\lim_{k\to \infty}\mu_k(\theta)=-\infty$, and therefore, by \eqref{MV-asympt},
\begin{equation}
\label{MV-alg-asymp}
M_k\big(T\mu_k(\theta)\big)\sim\frac{1}{2T|\mu_k(\theta)|}
\asymp k^{-\beta},\
V_k\big(T\mu_k(\theta)\big)\sim\frac{1}{2|T\mu_k(\theta)|^3}
\asymp k^{-3\beta}.
\end{equation}
 Let
$$
\gamma_1=2\alpha_1-\alpha-\beta,\ \ \gamma_2=2\beta_1-\beta,
\ \gamma_{12}=\alpha_1-\alpha+\beta_1-\beta.
$$
We have (see \eqref{eqconst} for the definition of $\eqconst$)
\begin{align}
\label{alg-pr-u-asym}
&\tau_k^2\E\int_0^T u_k^2(t)dt \eqconst k^{\gamma_1},\
\tau_k^4\Var \int_0^T u_k^2(t)dt \eqconst k^{2\gamma_1-\beta},\\
\label{alg-pr-v-asym}
&\nu_k^2\E \int_0^T v_k^2(t)dt \eqconst k^{\gamma_2},\
\nu_k^4\Var \int_0^T v_k^2(t)dt \eqconst k^{2\gamma_2-\beta},\\
\label{alg-pr-uv-asym}
&|\nu_k\tau_k|\,\E u_k^2(T) \eqconst k^{\gamma_{12}},\
\nu_k^2\tau_k^2\Var u_k^2(T) \eqconst k^{2\gamma_{12}},
\end{align}
and therefore
\begin{equation}
\label{alg-pr-Psi-asym}
\Psi_{1,N} \eqconst \begin{cases}
\mathrm{const.},& {\rm if}\ \gamma_1<-1,\\
\ln N,& {\rm if\ } \gamma_1=-1,\\
N^{\gamma_1+1},& {\rm if\ } \gamma_1>-1,
\end{cases}
\qquad
\Psi_{2,N} \eqconst \begin{cases}
\mathrm{const.},& {\rm if}\ \gamma_2<-1,\\
\ln N,& {\rm if\ } \gamma_2=-1,\\
N^{\gamma_2+1},& {\rm if\ } \gamma_2>-1,
\end{cases}
\end{equation}
\begin{equation}
\label{alg-pr-Psi12-asym}
|\Psi_{12,N}| \eqconst \begin{cases}
\mathrm{const.},& {\rm if}\ \gamma_{12}<-1,\\
\ln N,& {\rm if\ } \gamma_{12}=-1,\\
N^{\gamma_{12}+1},& {\rm if\ } \gamma_{12}>-1.
\end{cases}
\end{equation}
Next, we show that condition \eqref{alg-order-cond1} implies
\begin{equation}
\label{pr-alg1-2}
\lim_{N\to \infty} \frac{K_{1,N}}{\Psi_{1,N}}=1 \ {\rm with \ probability \ one},
\end{equation}
condition \eqref{alg-order-cond2} implies
\begin{equation}
\label{pr-alg1-4}
\lim_{N\to \infty} \frac{K_{2,N}}{\Psi_{2,N}}=1 \ {\rm with \ probability \ one},
\end{equation}
and either \eqref{alg-order-cond1} or \eqref{alg-order-cond2},
\begin{equation}
\label{pr-alg1-3}
\lim_{N\to \infty}D_N=0 \ {\rm with \ probability \ one.}
 \end{equation}
 Indeed, convergence \eqref{pr-alg1-2} follows from \eqref{main-asympt-u} and
\eqref{ap2}, because \eqref{alg-order-cond1} implies
$$
\sum_n \frac{n^{2\gamma_1-\beta}}{\Psi_{n,1}^2} < \infty.
$$
Similarly,  \eqref{pr-alg1-4} follows from \eqref{main-asympt-v} and
\eqref{ap2}, because \eqref{alg-order-cond2} implies
$$
\sum_n \frac{n^{2\gamma_2}}{\Psi_{n,2}^2} < \infty.
$$
For \eqref{pr-alg1-3}, we first observe that $\lim_{N\to \infty} K_{12,N}/\Psi_{12,N}$ exists with probability one. If $\gamma_{12}<-1$,
 the the limit is  a $\mathbb{P}$-a.s finite random variable. If  $\gamma_{12}\geq -1$, the \eqref{main-asympt-u0} and \eqref{ap2} imply that  limit is $1$.
  Then direct analysis shows that
$$
\lim_{N\to \infty} \frac{\Psi_{12,N}^2}{\Psi_{1,N}\Psi_{2,N}}=0
$$
if at least one of $\Psi_{1,N}$, $\Psi_{2,N}$ is unbounded.

Next, we show that \eqref{alg-order-cond1} implies
\begin{equation}
\label{pr-alg1-1}
\lim_{N\to \infty} \frac{\iota_{1,N}}{\Psi_{1,N}}=0\ {\rm with \ probability \ one},
\end{equation}
 and
\begin{equation}
\label{pr-alg2}
\lim_{N\to \infty} \frac{\iota_{1,N}}{\sqrt{\Psi_{1,N}}}=\xi_1 {\rm \ in \ distribution,}
\end{equation}
 whereas \eqref{alg-order-cond2} implies
\begin{equation}
\label{pr-alg2-1}
\lim_{N\to \infty} \frac{\iota_{2,N}}{\Psi_{2,N}}=0 \ {\rm with \ probability \ one},
\end{equation}
 and
\begin{equation}
\label{pr-alg2-2}
\lim_{N\to \infty} \frac{\iota_{2,N}}{\sqrt{\Psi_{2,N}}}=\xi_2
{\rm \ in \ distribution.}
\end{equation}
Indeed, \eqref{pr-alg1-1} follows from \eqref{main-asympt-u} and
\eqref{ap1}, because
\eqref{alg-order-cond1} implies that $\sum_{k} k^{\gamma_1}=+\infty$.
Similarly,  \eqref{pr-alg1-2} follows from
\eqref{main-asympt-v} and \eqref{ap1}.

Both \eqref{pr-alg2} and \eqref{pr-alg2-2} follow from Corollary \ref{C:3}.
Together with \eqref{pr-alg1-3}, the same Corollary also implies
independence of $\xi_1$ and $\xi_2$ if both \eqref{alg-order-cond1} and
\eqref{alg-order-cond2} hold.

To complete the proof of the theorem it remains  to show that
\begin{equation}
\label{pr-alg5-1}
\lim_{N\to \infty}\frac{\iota_{i,N} K_{12,N}}{K_{1,N}K_{2,N}}=0,\ i=1,2,
{\rm \ with \ probability \ one},
\end{equation}
 and
\begin{equation}
\label{pr-alg5-2}
\lim_{N\to \infty} \frac{\sqrt{\Psi_{i,N}}\,\iota_{i,N}K_{12,N}}{K_{1,N}K_{2,N}}=0, \  i=1,2,
{\rm \ in \ probability.}
\end{equation}
We leave to the interested reader to verify that \eqref{ap1} implies \eqref{pr-alg5-1}, and \eqref{pr-alg1-3}, \eqref{pr-alg2},
\eqref{pr-alg2-2} imply \eqref{pr-alg5-2}.

This completes the proof of Theorem \ref{th:main-alg}.
\end{proof}

\begin{remark} From Proposition \ref{prop:necessary}, we
see that condition \eqref{alg-order-cond1}
is  both necessary and sufficient for consistency and
asymptotic normality of estimator $\hat{\theta}_{1,N}$. Similarly,
condition \eqref{alg-order-cond2} is necessary and sufficient for consistency and
asymptotic normality of estimator $\hat{\theta}_{2,N}$.
\end{remark}

Since in the algebraic case the sum $\sum_{k=1}^N k^{\gamma}$
appears frequently, we introduce a special notation to describe the
asymptotic of this sum as $N\to \infty$ for $\gamma\geq -1$:
\begin{equation}
\label{sum-asymp}
\Upsilon_N(\gamma)=
\begin{cases}
N^{\gamma+1},& {\rm if} \ \gamma>-1,\\
\ln N,& {\rm if} \ \gamma=-1.
\end{cases}
\end{equation}
With this notation, 
$\sum_{k=1}^N k^{\gamma}
 \asymp \Upsilon_N(\gamma)$, $\gamma\geq -1$.

Let us consider several examples, in which $\boldsymbol{\Delta}$ is
the Laplace operator in a smooth bounded domain $G$ in $\R^d$ with zero
boundary conditions; $H=L_2(G)$. We start with these three equations:
\begin{align}
\label{alg-ex1}
&u_{tt}=\theta_1\boldsymbol{\Delta} u + \theta_2u_t + \dot{W}, \
\theta_1>0,\ \theta_2\in \R; \\
\label{alg-ex2}
&u_{tt}=\boldsymbol{\Delta} (\theta_1u+\theta_2u_t)+\dot{W}, \
\theta_1>0,\ \theta_2>0;\\
\label{alg-ex3}
&u_{tt}=\theta_1\boldsymbol{\Delta} u-\theta_2\boldsymbol{\Delta}^2u_t+\dot{W}, \ \theta_1>0,\ \theta_2>0.
\end{align}
The following table summarizes the results:

\begin{center}
\begin{tabular}{c c c c}
\hline
Asymptotic & Eq. \eqref{alg-ex1} & Eq. \eqref{alg-ex2} & Eq. \eqref{alg-ex3} \\
\hline
\hline
& & & \\
 $\Psi_{1,N}$ & $N^{\frac{2}{d}+1}$&$ N$ & $ \Upsilon_N(-2/d),\ d\geq 2$ \\
 & & & \\
 \hline
 & & & \\
$\Psi_{2,N}$ & $N$ &$ N^{\frac{2}{d}+1}$  &   $ N^{\frac{4}{d}+1}$  \\
& & & \\
\hline
\end{tabular}
\end{center}

In equations \eqref{alg-ex1}--\eqref{alg-ex3}, 
$\A_1$ and $\B_1$ are {\em leading}
operators, that is, $\alpha=\alpha_1$ and $\beta=\beta_1$.
This, in particular, ensures that the estimator
$\hat{\theta}_{2,N}$ is always consistent.

Let us now consider examples when $\A_1$ and $\B_1$ are not
the leading operators:
\begin{align}
\label{alg-ex4}
&u_{tt}=\big(\boldsymbol{\Delta} u+\theta_1u\big) +\big(\boldsymbol{\Delta} u_t +\theta_2u_t\big) + \dot{W}, \
\theta_1\in \R,\ \theta_2\in \R; \\
\label{alg-ex5}
&u_{tt}+\big(\boldsymbol{\Delta}^2u+\theta_1 u\big)= \big(\theta_2\boldsymbol{\Delta} u_t-\boldsymbol{\Delta}^2u_t\big)+\dot{W},
\theta_1\in \R,\ \theta_2\in R;\\
\label{alg-ex6}
&u_{tt}+\big(\boldsymbol{\Delta}^2 u+ \theta_1\boldsymbol{\Delta}u\big)=
\big(\theta_2u_t-\boldsymbol{\Delta}^2u_t\big)+\dot{W}, \ \theta_1\in \R,\ \theta_2\in \R.
\end{align}
The following table summarizes the results:

\begin{center}
\begin{tabular}{c c c c}
\hline
Asymptotic & Eq. \eqref{alg-ex4} & Eq. \eqref{alg-ex5} & Eq. \eqref{alg-ex6} \\
\hline
\hline
& & & \\
 $\Psi_{1,N}$ & $\Upsilon_N(-4/d),\ d\geq 4$&$ \Upsilon_N(-8/d),\ d\geq 8$ & $ \Upsilon_N(-4/d),\ d\geq 4$ \\
 & & & \\
 \hline
 & & & \\
$\Psi_{2,N}$ & $\Upsilon_N(-2/d),\ d\geq 2$ &$ N$  &   $ \Upsilon_N(-4/d),\ d\geq 4$  \\
& & & \\
\hline
\end{tabular}
\end{center}
As was mentioned in the Introduction, an interested reader can investigate a four-parameter estimation problem, such as
$$
u_{tt}+\big(\theta_{11}\boldsymbol{\Delta}^2 u+ \theta_{12}\boldsymbol{\Delta}u\big)
=\big(\theta_{21}u_t-\theta_{22}\boldsymbol{\Delta}^2u_t\big)+\dot{W}.
$$

\section{Analysis of Estimators: General Case}
\label{sec:GC}

While the algebraic case, corresponding to elliptic partial differential operators,
 seems the most natural, we believe that a more general case, allowing
 eigenvalues such as $\lambda_k\sim e^k$ or $\mu_k\sim \ln k$,
 is also worth considering, not only as a mathematical curiosity, but also as an
 example of a model with  observations coming from independent but not
 identical channels (see Korostelev and Yin \cite{KY}).

 As the proof of Theorem  \ref{th:main-alg} shows, the key arguments involve
 a suitable law of large numbers.  Verification of the corresponding 
 conditions is straightforward in the algebraic case, but is impossible 
  in the general case unless we make additional assumptions about the 
  eigenvalues of the operators. Indeed, as we work with weighted 
  sums of independent random variables, we need some conditions on the 
  weights for a law of large numbers to hold. In particular, the 
  weights should not grow too fast: if $\xi_k,\ k\geq 1,$ are 
  iid standard Gaussian random variables, then the sequence 
  $\{n^{-2}\sum_{k=1}^n n\xi_k^2,\ n\geq 1\}$ converges with 
  probability one to $1/2$, but 
  $\{e^{-n}\sum_{k=1}^n e^k\xi_k^2,\ n\geq 1\}$ does not have a limit, even in probability.

 Theorem \ref{T:3-28} in Appendix summarizes some of the laws of 
 large numbers, and leads to the following
\begin{definition}
 The sequence $\{a_n,\ n\geq 1\}$ of positive numbers  is  called
 {\tt slowly increasing} if 
 \begin{equation}
 \label{slowly-incr-s}
 \lim_{n\to \infty}\frac{\sum_{k=1}^na^2_n}{\Big(\sum_{k=1}^n a_k\Big)^2}=0.
 \end{equation}
 \end{definition}
 
 The purpose of this definition is to simplify the statement 
 of the main theorem (Theorem \ref{th:main-general} below). 
 It was not necessary in the algebraic case because the sequence $\{n^{\gamma},\ n\geq 1\}$ is slowly increasing if and only if $\gamma\geq -1$.
 The reason for the terminology is that the sequence $\{e^{n^{r}},\ n\geq 1\}$ 
 has  property \eqref{slowly-incr-s} if and only if $r<1$. 
 Further discussion of \eqref{slowly-incr-s}, including the connections with the 
  weak law of large numbers, is after the proof of Theorem \ref{T:3-28}.
 
  In general, we have to replace
\eqref{alg-order-cond1} with

\begin{center}
{\bf Condition 1.} The sequence
$\{ \tau_k^2M\big(T\mu_k(\theta_2)\big)/\lambda_k(\theta_1),\ k\geq 1\}$
    is slowly increasing,
\end{center}

 \noindent and   \eqref{alg-order-cond2}, with

\begin{center}
{\bf Condition 2.} The sequence
 $\{ \nu_k^2M\big(T\mu_k(\theta_2)\big),\ k\geq 1\}$
    is slowly increasing.
\end{center}

\begin{theorem}
\label{th:main-general}
Assume that
equation \eqref{E:3-151} is diagonalizable and hyperbolic.
\begin{enumerate}
\item If Condition 1 holds, then
\begin{align}
\label{gen-res-asymp1}
\lim_{N\to \infty}\hat{\theta}_{1,N}=\theta_1 \ \ {\rm in\ probability}; \\
\label{gen-res-asymp2}
\lim_{N\to \infty}\sqrt{\Psi_{1,N}}\Big(\hat{\theta}_{1,N}-\theta_1 \Big)
= \xi_1 \ {\rm in \ distribution,}\
\end{align}
where $\xi_1$ is a standard Gaussian random variable.
\item If Condition 2 holds
  then
\begin{align}
\label{gen-res-asymp3}
\lim_{N\to \infty}\hat{\theta}_{2,N}=\theta_2 \ \ {\rm in\ probability}; \\
\label{gen-res-asymp4}
\lim_{N\to \infty}\sqrt{\Psi_{2,N}}\Big(\hat{\theta}_{2,N}-\theta_2 \Big)
= \xi_2 \ {\rm in \ distribution,}\
\end{align}
where $\xi_2$ is a standard Gaussian random variable.
\item If both Conditions 1 and 2 hold,
then the random variables $\xi_1, \xi_2$ are independent.
\end{enumerate}
\end{theorem}

\begin{proof}
The main steps are the same as in the algebraic case
(Theorem \ref{th:main-alg}). In particular,
\eqref{pr-alg1-1} and \eqref{pr-alg2-1} continue to hold as long
as $\Psi_{1,N}\to \infty$ and $\Psi_{2,N}\to \infty$, respectively.
The only difference is that
 Conditions 1 and  2 do not provide enough information
about the almost sure behavior of $K_{12,N}/\E K_{12,N}$, and, 
in this general setting, there is no natural condition that would do that.
As a result,  in \eqref{pr-alg1-3},
 the convergence is in probability rather
than with probability one, and then, in both \eqref{pr-alg1-2}
and \eqref{pr-alg1-4}, convergence in probability will suffice.
Conditions 1 and 2 ensure 
\eqref{pr-alg1-2} and \eqref{pr-alg1-4},
respectively, but with convergence in probability rather than almost sure. 
This is a direct consequence of the weak law of large numbers. 

In the case of \eqref{pr-alg1-3}, we have
$$
\E |K_{12,N}| \leq \sum_{k=1}^N |\tau_k\nu_k|\, \E u_k^2(T)
$$
and, for all sufficiently large $k$,
$$
\E u_k^2 (T) \leq \frac{4T}{\lambda_k(\theta_1)} M\big(T\mu_k(\theta_2)\big)
\Big(1+\max\big(0,T\mu_k(\theta_2)\big)\Big),
$$
because  $xe^x-x \leq 4(e^x-x-1)(1+\max(0,x))$ for all $x\in \R$.
Then
\begin{equation}
\label{gen-aux-12}
\lim_{N\to \infty} \frac{\E |K_{12,N}|}{\sqrt{\Psi_{1,N}\Psi_{2,N}}}=0.
\end{equation}
Indeed, under Condition 1, \eqref{gen-aux-12} follows from
$$
\frac{\big(\E |K_{12,N}|\big)^2}{{\Psi_{1,N}\Psi_{2,N}}} \leq
\frac{16\sum_{k=1}^N
\frac{\tau_k^2M\big(T\mu_k(\theta_2)\big)}{\lambda_k(\theta_1)}
\ \frac{ \big(1+\max\big(0,T\mu_k(\theta_2)\big)\big)^2}{\lambda_k(\theta_1)}}
{T^2\sum_{k=1}^N
\frac{\tau_k^2M\big(T\mu_k(\theta_2)\big)}{\lambda_k(\theta_1)}}
$$
(Cauchy-Schwartz inequality) and
\begin{equation}
\label{pr-gen-1}
\lim_{k\to \infty} \frac{ \big(1+\max\big(0,T\mu_k(\theta_2)\big)\big)^2}{\lambda_k(\theta_1)}=0
\end{equation}
(hyperbolicity condition),
while, under Condition 2, \eqref{gen-aux-12} follows from
$$
\frac{\big(\E |K_{12,N}|\big)^2}{{\Psi_{1,N}\Psi_{2,N}}} \leq
\frac{16\sum_{k=1}^N
{\nu_k^2M\big(T\mu_k(\theta_2)\big)}
\ \frac{  \big(1+\max\big(0,T\mu_k(\theta_2)\big)\big)^2}{\lambda_k(\theta_1)}}
{T^2\sum_{k=1}^N
{\nu_k^2M\big(T\mu_k(\theta_2)\big)}}
$$
(Cauchy-Schwartz inequality with a different arrangement of terms)
and \eqref{pr-gen-1}.

The interested reader can fill in the details in the rest of the proof.
\end{proof}

As an example, consider the operators  with eigenvalues 
$\kappa_k=e^{2k}$, $\tau_k=e^k$, $\rho_k=0$, $\nu_k=\ln\ln(k+3)$
and assume that $\theta_1>0$, $\theta_2>0$.  Then 
$$
\lambda_k=e^{2k}+\theta_1e^k,\ \mu_k=\theta_2\ln\ln (k+3),
$$ 
so that $\tau_k^2/\lambda_k \sim 1$. Next, for all sufficiently large $k$, 
$$
\big(\ln (k+3)\big)^{T\theta_2/2}< M(T\mu_k) < \big(\ln (k+3)\big)^{T\theta_2},
$$
and also  $\nu_k^2M(T\mu_k) \asymp \big(\ln (k+3)\big)^{T\theta_2}$. 
Using integral comparison, we conclude that, for all $r>0$, 
 $$
 \sum_{k=1}^N \big(\ln k\big)^r
\sim N\big(\ln N\big)^r.
$$
 Thus, both Condition 1 and Condition 2 hold. 
By Theorem \ref{th:main-general}, both $\hat{\theta}_{1,N}$ 
and $\hat{\theta}_{2,N}$ are consistent and asymptotically 
normal. An interested reader can verify that 
$$
\Psi_{1,N}\asymp \frac{N\big(\ln N\big)^{T\theta_2}}{\big(\ln \ln N\big)^2},\ \
\Psi_{2,N}\asymp N\big(\ln N\big)^{T\theta_2}.
$$

\section{Acknowledgement} The work of both authors
 was partially supported
by the NSF Grant DMS-0803378.

\section*{Appendix}
\renewcommand{\theequation}{A.\arabic{equation}}
\setcounter{equation}{0}

\renewcommand{\thetheorem}{A.\arabic{theorem}}
\setcounter{theorem}{0}

First, let us recall some basic facts about second-order stochastic
ordinary differential equations with constant coefficients.
Consider the initial value problem
\begin{equation}
\label{ap-ode1}
\ddot{y}(t)-2b\dot{y}(t)+a^2y(t)=0,\ \ y(0)=0,\ \ \dot{y}(0)=1.
\end{equation}
With $2b=\mu_k(\theta_2)$ and $a^2=\lambda_k(\theta_1)$,
we recover \eqref{fund-sol}; recall that $\lambda_k(\theta_1)>0$
for all sufficiently large $k$. If $\ell=\sqrt{|b^2-a^2|}$, then
\begin{equation}
\label{ap-ode2}
y(t)=
\begin{cases}
\ds \frac{\sin(\ell t)}{\ell}\ e^{bt},&\ a^2>b^2; \\
\ds \ \ \ \ \ te^{bt},& a^2=b^2;\\
\ds \frac{\sinh(\ell t)}{\ell}\ e^{bt},& a^2<b^2;
\end{cases}
\end{equation}
as usual, $\sinh x = (e^x-e^{-x})/2$. Note that if $b<0$ and $b^2>a^2$, then $b+\ell<0$. The solution of the inhomogeneous equation 
$$
\ddot{x}(t)-2b\dot{x}(t)+a^2x(t)=f(t),\ \ x(0)=\dot{x}(0)=0
$$
is then $x(t)=\int_0^t y(t-s)f(s)ds$. 

Next, we formulate the  laws of large numbers and the central
limit theorem used in the proof of consistency and asymptotic normality of the
estimators.

To begin, let us recall Kolmogorov's strong law of large numbers.
 \begin{theorem}
 \label{prop-KSLLN}
  Let $\{\xi_k,k\geq 1\}$ be a sequence of independent random variables
  with $\E \xi_n^2<\infty$. If $\{b_n\, \n\geq 1\}$ is an unbounded increasing sequence of real numbers ($b_n\nearrow +\infty$) and \\
 $\sum_{n\geq 1}b_n^{-2} \Var(\xi_n)<\infty$, then
 \begin{equation}
 \label{ap-lln0}
 \lim_{n\to \infty}b_n^{-1}\sum_{k=1}^n (\xi_k-\E \xi_k)=0
  \end{equation}
  with probability one.
 \end{theorem}
 \begin{proof} See, for example, Shiryaev \cite[Theorem IV.3.2]{Shir}.
 \end{proof}

The following  laws of large numbers, both strong and weak, are often used in
 the current paper.

\begin{theorem}[Several Laws of Large Numbers]
\label{T:3-28}
Let $\chi_k, \ k\geq 1,$ be independent random variables, each with zero mean
 and positive finite variance. If
 \begin{equation}
 \label{ap0}
 \sum_{k\geq 1} \E\chi_k^2=+\infty,
 \end{equation}
 then
 \begin{equation}
\label{ap1}
\lim_{N\to \infty} \frac{\sum_{k=1}^N\chi_k}{\sum_{k=1}^N \E \chi_k^2}=0
\ {\rm with\  probability\  one}.
\end{equation}

Next, assume in addition that
\begin{equation}
\label{ap01}\E \chi_k^4 \leq c_1\Big(\E\chi_k^2\Big)^2
\end{equation}
for all $k\geq 1$,
with $c_1>0$ independent of $k$. Then
\begin{equation}
\label{slow0s}
\lim_{n\to \infty}
\frac{\sum_{k=1}^n\Big(\E\chi_k^2\Big)^2}{\Big(\sum_{k=1}^n \E\chi_k^2
\Big)^2}=0
\end{equation}
 implies
\begin{equation}
\lim_{N\to \infty} \frac{\sum_{k=1}^N\chi_k^2}
{\sum_{k=1}^N \E \chi_k^2}=1\ {\rm in \ probability}
\label{ap2s}
\end{equation}
and
\begin{equation}
\label{slow0}
\sum_{n\geq 1} \frac{\Big(\E\chi_n^2\Big)^2}{\Big(\sum_{k=1}^n \E\chi_k^2
\Big)^2}<\infty,
\end{equation}
 implies
\begin{equation}
\lim_{N\to \infty} \frac{\sum_{k=1}^N\chi_k^2}
{\sum_{k=1}^N \E \chi_k^2}=1\ {\rm with \ probability\ one}.
\label{ap2}
\end{equation}
\end{theorem}

\begin{proof}  To prove \eqref{ap1},
we take $\xi_n=\chi_n$ and $b_n=\sum_{k=1}^n \E\chi_k^2$ and apply
Theorem  \ref{prop-KSLLN}; note that
  convergence of
 $\sum_n b_n^{-2}\E \chi_n^2$ follows from divergence of
  $\sum_{k\geq 1} \E \chi_k^2$:
  $$
  \sum_n \frac{\E \chi_n^2}{b_n^{2}} \leq  \sum_n\left( \frac{1}{b_{n-1}}-\frac{1}{b_n}\right).
  $$
To prove \eqref{ap2}, we take $\xi_n=\chi_n^2$ and $b_n=\sum_{k=1}^n \E\chi_k^2$, and again apply Theorem \ref{prop-KSLLN}; this time,
we have to {\em assume}  convergence of the series
$\sum_n \Var{\xi_n}b_n^{-2}$. Finally, \eqref{ap2s} follows from
\eqref{slow0s} and Chebyshev's inequality.
\end{proof}

In other words, normalizing a sum of zero-mean random variables by the
total {\em variance} will give in the limit zero with probability one as ling as the
total variance is unbounded, while normalizing a sum of positive
random variables by the total {\em mean} will give in the limit
 one only under some
additional assumptions. Given a collection of iid standard normal random variables
$\{\xi_k,\ k\geq 1\}$, an interested reader can verify that the
sequence $\big(\sum_{k=1}^n e^k\xi_k^2\big)/\big(\sum_{k=1}^n e^k\big)$
does not converge  in probability as $n\to \infty$.

To understand the meaning of conditions \eqref{slow0s} and \eqref{slow0},
note that if $\xi_k, \ k\geq 1,$ are iid non-negative random variables with
 $\E \xi_1=A>0 $, then,  taking in Theorem \ref{prop-KSLLN}
 $b_n=\sum_{k=1}^n \E \xi_k=An$, we recover the classical
 strong law of large numbers:
 $$
 \lim_{n\to \infty} \frac{1}{n} \sum_{k=1}^n \xi_k=A
 $$
 with probability one. In the second part of Theorem \ref{T:3-28},
 we want to establish a similar result when the random variables $\xi_k$ are
  positive and independent, but not identically distributed.
  Condition \eqref{ap01} (which holds, for example, for Gaussian random
  variables)
  allows us to apply  Theorem \ref{prop-KSLLN} with  $b_n=\sum_{k=1}^N\E \xi_k$.  If
  $a_k:=\E \xi_k>0$ for all $k$, then conditions \eqref{ap01} and \eqref{slow0}  become, respectively,
  \begin{align}
 \label{slowincr1}
 &\sum_{n\geq 1} a_n=+\infty,\\
 \label{slowincr2}
 &\sum_{n\geq 1} \frac{a_n^2}{\left(\sum_{k=1}^n a_k\right)^2} < \infty.
 \end{align}
 On the other hand, if
 \begin{equation}
 \label{slowincr2s-a}
 \lim_{n\to \infty} \frac{\sum_{k=1}^na_k^2}{\left(\sum_{k=1}^n a_k\right)^2}=0,
 \end{equation}
 then  Chebyshev's  inequality leads to a weak law of large numbers.

    In general,  \eqref{slowincr1} does not imply
    \eqref{slowincr2} or \eqref{slowincr2s-a}
 (take $a_n=e^n)$, nor does  \eqref{slowincr2} imply \eqref{slowincr1} (take $a_n=1/n^2$), but obviously \eqref{slowincr2s-a} implies \eqref{slowincr1}.
 An interested reader can also verify that the sequence $\{e^{\sqrt{n}},\ n\geq 1\}$ satisfies \eqref{slowincr2s-a} but not \eqref{slowincr2}. On the other hand, 
 since we use \eqref{slowincr1} and \eqref{slowincr2} to prove a strong law of 
 large numbers, and use \eqref{slowincr2s-a} to prove a weak law 
 of large numbers, it will be natural to expect that conditions \eqref{slowincr1} and \eqref{slowincr2} together are stronger than \eqref{slowincr2s-a}. 
 Kronecker's Lemma (see \cite[Lemma IV.3.2]{Shir} with 
 $b_n=\left(\sum_{k=1}^n a_n\right)^2,\ x_n=a_n^2/b_n$) shows that 
 this is indeed the case: \eqref{slowincr1} and \eqref{slowincr2} imply \eqref{slowincr2s-a}. 
 
 We say that a sequence of positive numbers $\{a_n, \ n\geq 1\}$
 is {\tt slowly increasing} if  condition \eqref{slowincr2s-a} holds.  The notion of a slowly increasing sequence simplifies the conditions for consistency and asymptotic normality of the estimators in the general (non-algebraic) setting. Related
 conditions in the context of the law of large numbers can be found, for example, in the paper \cite{WLLN}. If $a_n=n^{\gamma},\ \gamma \in \R$ (algebraic case), then  \eqref{slowincr1} (that is, $\gamma\geq -1$) implies   \eqref{slowincr2},
 which is the reason for the strong consistency in Theorem \ref{th:main-alg}.

The following theorem is  used to prove asymptotic normality of
the estimators.

\begin{theorem}[A Martingale Central Limit Theorem]
\label{T:1-3}
Let $M_{i,n}=M_{i,n}(t)$, $t\geq 0$, $n\geq1$, $i=1,2$,
 be two sequences of continuous
square--integrable martingales.  If, for some $T>0$,
\[
\lim_{n\to \infty}\frac{\langle M_{i,n}\rangle(T)}{\E \langle
M_n\rangle(T)}=1,\ i=1,2, \quad {\rm \ in\  probability},
\]
and
$$
\lim_{n\to \infty}\frac{\langle M_{1,n}, M_{2,n}\rangle(T)}
{\Big(\E \langle M_{1,n}\rangle(T)\Big)^{1/2}
\Big(\E \langle M_{2,n}\rangle(T)\Big)^{1/2}}=0 \quad {\rm \ in\  probability},
$$
then
\[
\lim_{n\to\infty}
\left(\begin{array}{l}
 M_{1,n}(T)\big(\E \langle M_{1,n}\rangle(T)\big)^{-1/2}\\
 M_{2,n}(T)\big(\E \langle M_{2,n}\rangle(T)\big)^{-1/2}
\end{array}
\right)
=\mathcal{N}(0,I) \quad {\rm \ in \ distribution},
\]
 where $\mathcal{N}(0,I)$ is a two-dimensional vector whose
 components are independent standard Gaussian random variables.
\end{theorem}
\begin{proof}
If $X_n$ and $X$ are continuous square-integrable martingales with values in
$\R^d$ such
that $X$ is a Gaussian process and $\lim_{n\to \infty}\langle
X_n\rangle(T)=\langle X\rangle(T)$ in probability, then $\lim_{n\to
\infty}X_n(T)=X(T)$ in distribution; recall that, for a vector-valued martingale
$X=(X^{(1)}, \ldots, X^{(d)})$,
 $\langle X \rangle(t)$ is the symmetric matrix with entries
  $\langle X^{(i)}, X^{(j)} \rangle (t)$. This is  one of the central limit theorems for
martingales;  see, for example, Lipster and Shiryaev \cite[Theorem
5.5.11]{LSh3}. The result now follows if we take
\[
X_n(t)=
\left(
\begin{array}{l}
X_{1,n}\\
X_{2,n}
\end{array}
\right),\ \
X(t)=
\left(
\begin{array}{l}
w_1(t)/\sqrt{T}\\
w_2(t)/\sqrt{T}
\end{array}
\right),
\]
where
$$
X_{i,n}=\frac{M_{i,n}(t)}{\big(\E \langle M_{i,n}\rangle(T)\big)^{1/2}}, \
i=1,2,
$$
and $w_1, w_2$ are independent standard Brownian motions.
\end{proof}

\begin{corollary}\label{C:3}
Let $f_{i,k}=f_{i,k}(t),$ $t\geq 0$, $i=1,2$, $k\geq 1$ be continuous,
square-integrable processes and $w_k=w_k(t)$
 be independent standard Brownian motions. 
 Define
 $$
 \eta_{i,N}=\frac{\sum_{k=1}^N\int_0^T
 f_{i,k}(t)dw_k(t)}{\left(\sum_{k=1}^N\E \int_0^Tf^2_{i,k}(t)dt
 \right)^{1/2}},\ i=1,2.
 $$
 If 
 \begin{align}
 &\lim_{N\to \infty}
 \frac{\sum_{k=1}^N\int_0^Tf^2_{i,k}(t)dt}{\sum_{k=1}^N\E \int_0^T
 f^2_{i,k}(t)dt}=1 \quad {\rm \ in\  probability,}\ \ {\rm and \ } \\
 &\lim_{N\to \infty}
 \frac{\sum_{k=1}^N\E \left|\int_0^T
 f_{1,k}(t)f_{2,k}(t)dt\right|}{\left(\sum_{k=1}^N\E \int_0^T
 f^2_{1,k}(t)dt \right)^{1/2} \left( \sum_{k=1}^N\E \int_0^T
 f^2_{2,k}(t)dt\right)^{1/2}}=0, \notag
 \end{align}
 then
 \[
 \lim_{N\to \infty}
 \left(
 \begin{array}{l}
 \eta_{1,N}\\
 \eta_{2,N}
 \end{array}
 \right)
 =\mathcal{N}(0,I) \quad {\rm \ in \ distribution},
 \]
 where $\mathcal{N}(0,I)$ is a two-dimensional vector whose
 components are independent standard Gaussian random variables.
\end{corollary}

\begin{proof}
This follows from Theorem \ref{T:1-3} by taking
\[
M_{i,n}(t)=\frac{\sum_{k=1}^n\int_0^t f_{i,k}(s)dw_k(s)}
{\left(\E \int_0^T
 f^2_{i,k}(t)dt\right)^{1/2}},
\]
because $\E \langle M_{i,n}\rangle(T)=1$ and 
$$
\langle M_{1,n}, M_{2,n}\rangle(T)= \frac{\sum_{k=1}^N \int_0^T
 f_{1,k}(t)f_{2,k}(t)dt}{\left(\sum_{k=1}^N\E \int_0^T
 f^2_{1,k}(t)dt \right)^{1/2} \left( \sum_{k=1}^N\E \int_0^T
 f^2_{2,k}(t)dt\right)^{1/2}}.
$$
\end{proof}


\def\cprime{$'$} \def\cprime{$'$} \def\cprime{$'$} \def\cprime{$'$}
  \def\cprime{$'$} \def\cprime{$'$}
\providecommand{\bysame}{\leavevmode\hbox to3em{\hrulefill}\thinspace}
\providecommand{\MR}{\relax\ifhmode\unskip\space\fi MR }
\providecommand{\MRhref}[2]{%
  \href{http://www.ams.org/mathscinet-getitem?mr=#1}{#2}
}
\providecommand{\href}[2]{#2}

\end{document}